\documentclass[twoside,a4paper,reqno,12pt]{amsart}
\usepackage{amsfonts, amsbsy, amsmath, amssymb, latexsym}
\usepackage{mathrsfs,array}
\usepackage[top=1 in,right=1 in,bottom=1 in,left=1 in]{geometry}
\usepackage[utf8]{inputenc}
\usepackage{hyperref}
\usepackage[english]{babel}
\usepackage{graphicx}
\usepackage{mathtools}
\usepackage{floatflt}
\usepackage{multicol}
\usepackage{lipsum}
\usepackage{diagbox}
\usepackage{wrapfig}
\usepackage{relsize}
\usepackage{booktabs}
\usepackage{pgfplots}
\usepackage{pgfplotstable}
\usepackage{enumerate}
\usepackage{blindtext}
\usepackage{amssymb}
\usepackage{amsthm}
\pgfplotsset{compat=1.18} 

\usepackage{array,tabularx}

\usepackage[english]{babel}

\newtheorem{theorem}{Theorem}[section]

\newtheorem{proposition}[theorem]{Proposition}
\newtheorem{corollary}[theorem]{Corollary}

\theoremstyle{definition}
\newtheorem{definition}[theorem]{Definition}

\theoremstyle{remark}
\newtheorem{remark}[theorem]{Remark}

\numberwithin{equation}{section}

\title{Poincaré series of semigroups}

\author{Antonio Campillo}
\address{Departamento de Álgebra, Análisis Matemático y Geometría y Topología Facultad de Ciencias Universidad de Valladolid\\
Paseo de Belén 7 47011 Valladolid\\
Spain}
\email{antonio.campillo@uva.es}
\thanks{}

\author{Raquel Melgar}
\address{LaBRI, Université de Bordeaux \\
Domaine Universitaire, 351, Cours de la Libération, 33405 Talence \\
France}
\email{raquel.melgar@labri.fr}

\date{}
\thanks{}

\begin{document}

\begin{abstract}
Many invariants of finitely generated positive cancelative commutative semigroups can be studied from their Poincaré series. We offer and present several closed formulas for them.  Moreover, those formulas have elementary proofs and are presented in two distinct forms: one characterized in terms of simplicial complexes and the other by a purely set theoretical approach.
\end{abstract}

\maketitle
\textit{Dedicated to Professors Alejandro Melle-Hernández and Enrique Artal Bartolo on the occasion of their 55th and 60th birthdays, respectively.}


\section{Introduction}
Algebraic Poincaré series of one or more variables are classically used in algebra as generating functions of invariants of graded structures; see, for instance \cite{MS}, \cite{P}. In the last 30 years, Poincaré series of several variables have been introduced in \cite{CDK}, \cite{CDGZ1}, \cite{CDGZ2} and used in geometry and topology to generate invariants in terms of multi-index filtrations, integrations with respect Euler characteristics, and calculations with Laurent Series. Those series , say topological Poincaré series, have been computed and applied to algebraic varieties and singularities, mainly curves, surfaces and toric varieties.

For affine toric varieties, i.e., those given algebraically by affine semigroups, it was proved in  \cite{L} that their Poincaré series, with respect their natural filtrations given by monomial valuations, are linear images of the classical generating functions, that we also call algebraic Poincaré series of affine semigroups. This poses the problem of studying the algebraic Poincaré series of the semigroup in detail and in depth, and obtaining formulas for it.

In this paper we study the algebraic Poincaré series of finitely generated positive cancelative commutative semigroups, namely of their semigroup algebras, following the philosophy and the techniques used in \cite{CG}, \cite{CP} and \cite{BCPV} for studying graded syzygies.

We offer and present closed formulas for this Poincaré series in different terms. On one hand, in terms of simplicial complexes associated to the semigroup elements and their versions in terms of the Alexander duals of those simplicial complexes; on the other hand, purely set theoretical ones, in terms of the monomial indicator series of certain key basic sets and their versions in terms of colored graphs.

We also study the combinatorial impact of those formulas on complete intersections, symmetrical semigroups, and the depth of the semigroup algebra. Since they are not needed for our purposes, in the article we have avoided syzygies, and, therefore, also their corresponding syzygy Poincaré series in an additional variable $v$. However, we also discuss how such Poincaré series is also explicitly determined by above key sets.

\section{Simplicial complexes and basic sets}
\subsection{Preliminaries}
Along this paper $S$ will stand for a finitely generated positive cancelative commutative semigroup. 

A semigroup is a set endowed with an inner operation $+$ having a $0$ element. The semigroup is commutative when $+$ is commutative, it is cancelative when $m+n=m'+n$ is only possible if $m=m'$, it is positive when $m+n=0$ is only possible if $m=n=0$ and it is finitely generated when finitely many of its elements generate the semigroup in the sense that any element is a linear combination with non negative integers as coefficients of the generators.

The group $G(S)$ generated by $S$ is the group of classes of pairs $(m,n) \in S \times S$ for the equivalence relationship $$(m,n)\sim (m',n')\quad \iff \quad m+n'=m'+n.$$

The cancelative property implies that the obvious semigroup homomorphism $S \longrightarrow G(S)$ is injective, and therefore $S$ can be viewed as a subsemigroup of $G(S)$. The positive condition means that $$S \cap (-S)=\{0\},$$ $S$ and $-S$ seen as subsets of $G(S)$. The commutative property implies that $G(S)$ is an abelian group and the finitely generated property implies that $G(S)$ is isomorphic to $\mathbb{Z}^d \times T$ where $T$ is the torsion subgroup of $G(S)$ and $d$ is a well defined integer called the rank of the group $G(S)$ and dimension of the semigroup $S$.

The cone $C(S)$ is the cone generated by the image of $S$ in $G(S)\otimes_{\mathbb{Z}} \mathbb{Q} \cong\mathbb{Q}^d$, i.e., the convex cone in $\mathbb{Q}^d$ generated by the halflines of directions $m\otimes 1$ for $m\in S$. The positive property is equivalent to $C(S)$  being a strongly convex cone, i.e., $C(S)\cap (-C(S))=\{0\}$ and $\mathbb{Q}^d$ being the $\mathbb{Q}$-linear space generated by $C(S)$, therefore $d$ is also the dimension of the cone $C(S)$. Moreover, $C(S)$ is a strongly convex rational polyhedral cone.

The positive property for finitely generated cancelative commutative semigroups is equivalent to the existence of a semigroup homomorphism $\lambda: S \longrightarrow \mathbb{N}$ such that $\lambda(m)=0$ iff $m=0$ and also to the fact that for each $m\in S$ there are only finitely many ways to write $m$ as a linear combination of given generators (see \cite{CP}, proposition 2.1).

Now, consider a symbol $t$ and another ones, called powers, $t^m$ for $m\in S$. The semigroup ring $R=\mathbb{Z}[S]$ is the set of expressions $$\sum_{m\in S}a_m t^m$$ $a_m \in \mathbb{Z}$ and where there are only finitely many nonzero $a_m$. The ring structure is given by the obvious addition and the multiplication determined by the rule $t^m \cdot t^n=t^{m+n}$.

In the same way, the power series ring $\hat{R}=\mathbb{Z}[[S]]$ is defined as above but allowing infinitely many non zero $a_m$ and taking into account that the rule $t^m \cdot t^n=t^{m+n}$ determines a well defined product because of the existence of the homomorphism $\lambda$, as it has finite fibers as a map. Also $\lambda$ allows to prove that the invertible series, say the units $s$ of the ring $\hat{R}$ are exactly those whose coefficient $a_0$ (say, the independent term) is $+1$ or $-1$. In particular, $1-t^m$ is invertible for all $m\in S\backslash \{0\}$ and its inverse element is the series $$\sum_{i=0}^\infty t^{im}.$$

For the abelian group $G(S)$, viewed as a semigroup, one has larger similar objects $L=\mathbb{Z}[G(S)]$ and $\hat{L}=\mathbb{Z}[[G(S)]]$ than $R$ and $\hat{R}$ now, allowing symbols $t^m$ for $m\in G(S)$. Both $L$ and $\hat{L}$ are additive abelian groups. $L$ is a ring, $\hat{L}$ is not a ring but it is a $L$-module.

Finally, the semigroup is said to be an affine semigroup when the torsion subgroup $T$ of $G(S)$ is trivial, i.e., $T=0$. Affine semigroups are given, in practice, as finitely generated positive subsemigroups $S$ of $\mathbb{Z}^h$ for an integer $h>0$. Thus, $G(S)$ is the group of lattice points of $\mathbb{Z}^h$ lying in the linear variety of $\mathbb{Q}^h$ spanned by $S$. If $d$ is the dimension of this variety, the torsion free group $G(S)$ is isomorphic to $\mathbb{Z}^d$, so $d$ is also the dimension of $S$.

If one chooses a basis of the group $G(S)$ as a free $\mathbb{Z}$-module, the symbol $t$ can be identified with a collection of variables $t_1,t_2,...,t_d$ and $t^m$ with the monomial $t_1^{m_1}\cdot t_2^{m_2}\cdot...\cdot t_d^{m_d}$ if $m=(m_1,m_2,...,m_d)\in G(S)$. Thus computations in $R$, $\hat{R}$, $L$ and $\hat{L}$ become computations with Laurent polynomials and series.

However, general finitely generated positive cancelative commutative semigroups are not affine. A typical example for it is the semigroup $S = \ell (\mathbb{N}^r)$, where $\ell: \mathbb{Z}^r\rightarrow G$ is a surjective group homomorphism, $G$ is an abelian group and $\mathbb{N}$ is the set on nonnegative integers in $\mathbb{Z}$. In this case, $G = G(S)$ is a finitely generated abelian group which has torsion in general. 

Also notice that other different notions of affine semigroups appear in the literature. For instance, in \cite{CC} affine semigroups, even topological affine semigroups are defined in the context of real topological vector spaces. But, these notions does not fit in our above context of finitely generated commutative semigroups.

\begin{definition}
The algebraic Poincaré series of the semigroup $S$, or simply the generating series of $S$, is the element of $\hat{R}$ given by $$P=P_S=\sum_{m\in S} t^m.$$
\end{definition}

The algebraic Poincaré series for the group $G(S)$ can be defined as the element $P_{G(S)}$ of $\hat{L}$ given by $$P_{G(S)}=\sum_{m\in G(S)}t^m.$$

\subsection{Formulas for Poincaré series}
In this subsection, we will compute the Poincaré series $P$ by two types of formulas, one in combinatorial terms and other in set theoretical ones. It means that one has a formula of each type for each choice of a non empty finite subset $E$ of $S\backslash\{0\}$. In practice, $E$ can be any subset satisfying those properties.

The key topological tool we consider are the abstract simplicial complexes $T_m$, $m\in S$, given by $$T_m=\{J\subset E\,|\, m-e_J\in S\}$$ where $e_J=\sum_{e\in J}e$.

Notice that $m-e_J$ takes sense in $G(S)$. The subsets $J$ in $T_m$ are called faces of dimension $\# J-1$ of $T_m$. Also notice that the empty set $\emptyset$ is a face of dimension $-1$ of $T_m$, and that for every choice of a field $\mathbb{K}$, it is defined the $\mathbb{K}$-vector space reduced simplicial homology $\tilde{H}_i(T_m)$ and cohomology $\tilde{H}^i(T_m)$ both of dimension $\tilde{h}_i(T_m)$, $i\geq -1$ which is 0 for $i\geq \# E -1$ (see for instance \cite{MS} page 9 for the precise definitions). The reduced homology $\tilde{h}_i(T_m)$ depends on the characteristic of the field, however the reduced Euler characteristic
$$\tilde{\chi}(T_m)=\sum_{i=-1}^{\#E-2}(-1)^i\tilde{h}_i(T_m)=\sum_{J\in T_m}(-1)^{\#J-1}$$ depends only on $S$ and not on the characteristic of the field. The subcomplexes $T_m$ were considered and used in \cite{H} and in \cite{CM} and later extensively used in \cite{CG}, \cite{CP}, \cite{BCPV} among others, in order to compute the syzygies of affine toric varieties. The simplicial complexes $T_m$ allow to use not only combinatorics, but also topology as their reduced homology coincides with the reduced singular homology of their topological realization in the euclidean space.

The key basic set theoretical tools we consider are the subsets $Q$, $D$ of $S$, dependent of $E$, given by
\begin{align*}
Q & =\{m\in S\, |\, m-e \notin S\text{ for all }e\in E \}\\
D & =\{m\in S\, |\, m-e_{supp(m)} \notin S\}
\end{align*}
where $supp(m)=\{e \in E\, |\, m-e \in S \}$.

More precisely, we will also consider the subsets $D^J$ of $S$ for $J \subset E$, $\# J\geq 2$, given by 
\begin{align*}
D^J & = \{m\in S\, |\, J \subset supp(m)\text{ and } m-e_J \notin S\}\\
	& = \{m\in S \, |\, J \subset supp(m)\text{ and } J \notin T_m\}.
\end{align*}

The set $D$ is the union of the sets $D^J$, $J\subset E$, $\#J \geq 2$. In practice, as done in \cite{CG}, this union is a good conceptual definition for $D$. Notice that $Q\cap D=\emptyset$, as the elements $m\in Q$ are characterized by the property $supp(m)=\emptyset$. Also they are characterized by the property $T_m=\{\emptyset\}$, i.e. $T_m$ is the simplicial complex having only the empty face. 

The set $Q$ is well known in the literature and it is usually called the Apéry set with respect of $E$. The sets $Q$, $D$ and $D^J$ were introduced in \cite{CG}.

When $\#E$ is small, for instance $\#E \leq 4$, the complexes $T_m$ are easily visualized and therefore their homology and $\tilde{\chi}$ are easy to compute. 

In  \cite{CG} and \cite{CP} it is proved that if $E$ generates the cone $C(S)$ then $\tilde{\chi}(T_m)\neq 0$ holds only for finitely many elements $m\in S$.

Finally for each subset $B\subset S$ define $N_B=\sum_{m\in B}t^m\in \hat{R}$ the monomial indicator, or simply indicator, of $B$. The main result of this section is the following one.

\begin{theorem}
\label{teorema_puntos}
With definitions and assumptions as above one has
\begin{enumerate}[itemsep=5mm]
\item $P=P_S=\dfrac{-\sum_{m \in S}\tilde{\chi}(T_m) t^m}{\prod_{e \in E}(1-t^e)}$
\item $P=P_S=\dfrac{N_Q-\sum_{J \subset E,\, \#J\geq 2}(-1)^{\#J}N_{D^J}}{\prod_{e \in E}(1-t^e)}$
\item If $E$ generates the cone $C(S)$ then $$\prod_{e \in E} (1-t^e) \cdot P = p \in R$$
\item If $q'$, $p'$ are elements of $R$ such that $q' \cdot P = p'$ then one has $p \cdot q'=p'\cdot q$, where $q = \prod_{e\in E}(1-t^e)\in R = \mathbb{Z}[S]$ is a unit in $\hat{R}=\mathbb{Z}[[S]]$.
\end{enumerate}
\end{theorem}
\begin{proof}
First, one has $$\tilde{\chi}(T_m)=\sum_{J\in T_m}(-1)^{\#J-1}.$$
Now, one has the equalities 
\begin{align*}
\prod_{e\in E}(1-t^e)P & = \sum_{J \subset E}(-1)^{\#J}\sum_{n\in S}
t^{n+e_J}\\
& = -\sum_{m\in S}(\sum_{J\in T_m}(-1)^{\#J-1})t^m= -\sum_{m\in S}\tilde{\chi}(T_m)t^m.
\end{align*}

This proves \textit{(1)} since the factors $1-t^e$ are units in $\hat{R}$ and the term $\frac{1}{1-t^e}$ stands for the inverse series of the unit $1-t^e$. 

Second, \textit{(2)} follows from the equality 
$$N_Q-\sum_{J\subset E,\, \#J\geq 2}(-1)^{\#J} N_{D^J}=-\sum_{m\in S}\tilde{\chi}(T_m)t^m.$$
To check this equality, one needs to see that the coefficients of $t^m$ in both sides are the same. 

If $m\notin Q$, $m \notin D^J$ for every $J\subset E$, $\#J \geq 2$, then $m-e_{supp(m)}\in S$, so $T_m$ is the full simplex in the vertices of $supp(m)$, hence $\tilde{\chi}(T_m)=0$ and both coefficients are equal to 0. If $m\in Q$, one has $m\notin D^J$ for all $J\subset E$, $\#J\geq 2$, since $Q\cap D=\emptyset$. Moreover, $T_m=\{\emptyset\}$, so $\tilde{\chi}(T_m)=-1$ and the coefficients of both sides are equal to 1. If $m\in D$, then $m\notin Q$. One has $\# supp (m)\geq 2$ and $m-e_{supp(m)}\notin S$. Thus, the coefficient on the left is equal to $$-\sum_{\substack{J\subset supp(m)\\ J\notin T_m}}(-1)^{\# J}$$ taking into account that $m\in D^J$ iff $J\notin T_m$. On the other hand, the obvious equality
$$\sum_{\substack{J\subset supp(m)\\ J\in T_m}}(-1)^{\#J}+\sum_{\substack{J\subset supp(m)\\ J\notin T_m}}(-1)^{\#J}=\sum_{J\subset{supp(m)}}(-1)^{\#J}=(1-1)^{\#supp(m)}=0$$ proves that the coefficient on the left is also equal to $$\sum_{\substack{J\subset supp(m)\\J\in T_m}}(-1)^{\#J}=-\tilde{\chi}(T_m)$$ and therefore equal to the coefficient on the right. This coefficient is deduced to be equal by using the Newton's binomial formula.

Statement \textit{(3)} follows from the algebraic fact that if $E$ generates the cone $C(S)$ then $\tilde{\chi}(T_m)\neq 0$ for only finitely many semigroup elements $m$. This was proved in detail in \cite{CG}. Thus the series $-\sum_{m\in S}\tilde{\chi}(T_m)t^m$ is an element of $R$.

Finally \textit{(4)}, follows from the elementary facts $q= \prod{e\in E}(1-t^e)\neq 0$, it is a unit of $\hat{R}$, and  $$q'p=(q'q)P=(qq')P=qp', $$ the three elements $p$, $q$ and $q'$ being in R and $p'$ being in $\hat{R}$.
\end{proof}

As said before, for semigroups the symbol $t$ is realized as a set of variables, $P$ is a power series, $p$, $q$, $p'$ and $q'$ are polynomials and $P=\frac{p}{q}= \frac{p'}{q'}$ is also a well defined rational function of $\mathbb{Q}[t]$. This fact is elementary but, a priori, nontrivial.

For general semigroups, $q$ is invertible in $\hat{R}$, so $P=\frac{p}{q}=p q^{-1}$ is also a equality of elements of $\hat{R}$. If $q'$ is invertible too, $P=\frac{p'}{q'}=p'q'^{-1}$ as elements of $\hat{R}$ as well. For us, in practice $$q'=\prod_{e'\in E'}(1-t^{e'})$$ for other subsets $E'$ of $S$, so the various expressions $P=\frac{p}{q}=\frac{p'}{q'}$ show how $p$ and $p'$ are related when $E'$ is replaced by $E$ or viceversa, i.e. how the data $Q$, $D^J$ and $\tilde{\chi}(T_m)$ are related for both choices. 

\subsection{Particular choices of $E$}
If $\#E=1$, $E=\{e\}$, the combinatorial tools $T_m$ have at most one vertex. So, there are only two possibilities: $T_m$ has only the empty face or it is a topological point; thus, $\tilde{\chi}(T_m)=-1$ in the first case and $\tilde{\chi}(T_m)=0$ in the second one. The only set theoretical tool is the Apéry set $Q=\{m\in S\, |,\, m-e\notin S\}$, and one has $T_m$ is the empty face for $m\in Q$ and a point for $m\notin Q$. The algebraic Poincaré series is just given by $$P=\frac{N_Q}{1-t^e}.$$

If $\#E=2$, $E=\{e, e'\}$, $T_m$ is either the empty face, one of the two possible points, both points together and isolated or a line segment joining both points. One has $\tilde{\chi}(T_m)=0$ in the second and fourth cases, $\tilde{\chi}(T_m)=-1$ in the first and $\tilde{\chi}(T_m)=1$ in the third one, the last two corresponding to $m\in Q$ and $m\in D=D^E$ respectively. Thus, the Poincaré series is now given by $$P=\frac{N_Q-N_D}{(1-t^e)(1-t^{e'})}.$$ 

\begin{figure}[h]
\includegraphics[width=50mm]{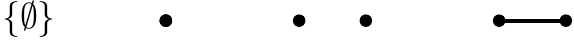}
\caption{Complexes $T_m$ for $\#E=2$}
\label{TME2}
\end{figure}

To give a more particular case, assume $\#E=3$, $E=\{e, e', e''\}$. Now the possibilities for $T_m$ are the empty face, one point, two isolated points, three isolated points, one segment, one segment and one isolated point, two segments sharing one of their extremes, the border of a triangle and a full triangle. Their respective values of $\tilde{\chi}(T_m)$ are $-1, 0, 1, 2, 0, 1, 0, -1$ and $0$. Now $m\in Q$ iff $T_m$ is the empty face, $m\in D$ iff $T_m$ it is not a full triangle nor a segment, $m\in D^{\{e, e'\}}\cap D^{\{e, e''\}}\cap D^{\{e',e''\}}$ iff $T_m$ consists of the three isolated points, but in general describing other topological possibilities in terms of $Q$ and $D^J$ become not so obvious. 

\begin{figure}[h]
\includegraphics[width=80mm]{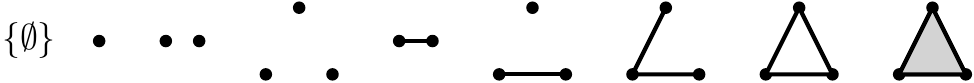}
\caption{\label{figure: two}Complexes $T_m$ for $\#E=3$}
\label{TME3}
\end{figure}

Now, formula \textit{(1)} of algebraic Poincaré series becomes shorter if one collects the terms of the sum having the same topology of $T_m$. Formula \textit{(2)} is given by $$P=\frac{N_Q+N_{D^E}-N_{D^{\{e,e'\}}}-N_{D^{\{e,e''\}}}-N_{D^{\{e',e''\}}}}{(1-t^e)(1-t^{e'})(1-t^{e''})}$$

Above formulas can be used to recover known results in the literature. For instance, last formula for $\#E=3$ allows to recover the results in  \cite{D}.

Above discussion shows how, in fact, one has many different possibilities to exhibit fractional expressions of cyclotomic polynomials of type $1-t^e$. Thus, as corollaries of Theorem \ref{teorema_puntos}, one can find nice practical consequences for the calculus. 

For instance, one of the consequences is formulated below which aims to reduce $\#E$ to 1. 

Let $E\subset S$ be a nonempty finite subset of $S$, $e_E=\sum_{e\in E}e$, and $Q_E=\{m\in S\,\,  |\, \, m-e_E\notin S\}$ the Apéry set of $S$ with respect the single element $e_E$. For each $J'\subset E$, $\# J' \geq 1$, consider the set $$E^{J'}=\{m\in S\, |\, m\notin Q_E, m-e_{J'}\in Q_E\}.$$

\begin{corollary}
With notations as above one has $$P=N_{Q_E}-\dfrac{\sum(-1)^{\#J'}N_{E^{J'}}}{\prod_{e\in E}(1-t^e)}$$ where the sum ranges over the subsets   $J'\subset E$ such that $\#J'\geq 1$.
\end{corollary}

\begin{proof}
Consider the choices $E$ and $E'=E\cup \{e_E\}$. Let $Q$, $D^J$ (resp. $Q', D'^{J}$), $\#J\geq 2$, the combinatorial tools for $E$ (resp. $E'$). One has $Q=Q'$ and $D^J=D'^J$ for $J\subset E$. Moreover, for $J=J'\cup \{e_E\}$, $J'\subset E$, $\#J\geq 1$ one has $D'^J=E^{J'}$. Then, the respective formulas \textit{(2)} for the choices $E'$ and $\{e^E\}$ provide the equality
$$\dfrac{N_{Q_E}}{1-t^{e_E}}=\dfrac{N_Q-\sum_{\substack{J\subset E\\ \#J\geq 2}}(-1)^{\#J}N_{D^J}+\sum_{\substack{J'\subset E\\ \#J'\geq 1}}(-1)^{\#J'}N_{E^{J'}}}{(1-t^{e_E})\prod_{e\in E}(1-t^e)}$$

Now, by multiplying both members by $(1-t^{e_E})$ one gets
$$N_{Q_E}=P+\dfrac{\sum_{\substack{J'\subset E\\ \#J'\geq 1}}(-1)^{\#J'}N_{E^{J'}}}{\prod_{e\in E}(1-t^e)}$$ which proves the result.
\end{proof}

\section{Colored graphs}
A more complete explanation of the result in section 2 can be done by means of colored graphs, a technique considered in \cite{CG} for computing the syzygies of various graded structures of the ring $R$. Here, our objective are algebraic Poincaré series, which is more modest than that of syzygies. However, we think that its explanation by means of colored graphs is interesting by itself.

Fix a finite subset $A$ of $S$ with $A\cap E = \emptyset$ and call colors to the elements $a$ of $A$. Now, consider subsets $B$ of $S$ satisfying the property that if $b\in B$ and $n,n'\in S$ satisfy $b+n+n'\in B$, then $b+n\in B$ and $b+n'\in B$. The key sets $Q$ and $D^J$ for $J\subset E$, $\#E\geq 2$ satisfy that property. Then one has an oriented simple colored graph $\mathcal{G}_B$ defined as follows. Vertices of  $\mathcal{G}_B$ are elements $m\in B$ such that $m-a_I\in B$ where $I\subset A$ and $a_I=\sum_{a\in I}a$. We have an oriented edge $(m,m')$ in $\mathcal{G}_B$ when $m'-m=a$ for some $a\in A$. Each edge is oriented and of color $a$. The graph is simple as $a$ is unique. The graph is infinite when $B$ is infinite. A set $I$ such that $m-a_I\in B$ is said to be a relation (reason in \cite{CG}) of dimension $\#I-1$ for $m$ to be a vertex of $\mathcal{G}_B$.

For $-1\leq \ell\leq \#A-1$ and $m\in S$, $C_\ell (\mathcal{G}_B, m)$ denote the free $\mathbb{Z}$-module generated by all the relations of dimension $\ell$, and $\partial_\ell$ the composition of the simplicial boundary $\delta_\ell$ of the full simplex $\mathcal{P}(A)$ of subsets (parts) of the set $A$, so with set of vertices $A$ with the obvious projection
$$C_{\ell-1}(\mathcal{P}(A))\longrightarrow C_{\ell-1}(\mathcal{G}_B,m).$$
Then, the property required to $B$ shows that $C_\bullet (\mathcal{G}_B, m)$ and $\partial_{\bullet}$ define a chain complex whose homology is denoted by $H_\ell(B,m)$, its rank by $h_\ell(B,m)$, its Euler characteristic by $$\chi(B,m)=\sum_{\ell=-1}^{\#A-1}(-1)^\ell h_\ell (B,m)=\sum_{\ell=-1}^{\#A-1}(-1)^\ell r_\ell (B,m)$$ where $r_\ell(B,m)$ is the number of relations of dimension $\ell$ for $m$ of $\mathcal{G}_B$, and the element $P_{\mathcal{G}_B}\in \hat{R}$ associated to $\mathcal{G}_B$ given by $$P_{\mathcal{G}_B}=\dfrac{-\sum_{m\in S}\chi(B,m)t^m}{\prod_{e\in E}(1-t^e)\prod_{a\in A}(1-t^a)}.$$

Next, we come to the main result of the section.

\begin{theorem}
\label{theorem: teorema 2}
With assumptions and notations as above, one has:
\begin{enumerate}[itemsep=5mm]
\item $P_{\mathcal{G}_B}=\dfrac{N_B}{\prod_{e\in E}(1-t^e)}$ and it does not depend on $A$.
\item  $P=P_S=P_{\mathcal{G}_Q}-\sum_{\substack{J\subset E \\ \#J \geq 2}}(-1)^{\#J}P_{\mathcal{G}_{D^J}}$.
\end{enumerate}
\end{theorem}
\begin{proof}
One has 
\begin{align*}
\prod_{a\in A}(1-t^a)N_B & =\sum_{\substack{I\subset A \\ b\in B}}(-1)^{\#I}t^{a_I+b}\\
&=-\sum_{m\in S}\sum_{\ell=-1}^{\#A-1}(-1)^\ell r_{\ell}(B,m)t^m=-\sum_{m\in S}\chi(B,m)t^m.
\end{align*}

This shows part \textit{(1)}. Part \textit{(2)} follows from formula \textit{(2)} in Theorem \ref{teorema_puntos}, taking into account that the key sets $Q$, $D^J$ satisfy the requirement to $B$.
\end{proof}

The homology of the graphs $\mathcal{G}_Q$ and $\mathcal{G}_{D^J}$ also allows us to compare the combinatorial data $T_m$ for the choice of $E$ and the corresponding simplicial complexes $T'_m$ for the choice of the larger set $E'=E\cup A$. For it, consider another colored graph $\mathcal{G}_{\bar{D}}$ on the key set $D$ defined as before but now taking as relations of dimension $\ell$ the sets $I\cup J$, $I\subset A$, $J\subset E$, $m-a_I\in D^J$ and $\#I\cup J =\ell+2$. Now, denote by $H_\ell(\mathcal{G}_{\bar{D}},m)$, $h_\ell(\mathcal{G}_{\bar{D}}, m)$ and $\chi(\overline{D}, m)$ the corresponding homology, ranks and Euler characteristic, the associated series to $\mathcal{G}_{\bar{D}}$ is defined by $$P_{\mathcal{G}_{\bar{D}}}=\dfrac{-\sum_{m\in S}\chi(\overline{D},m)t^m}{\prod
_{e\in E}(1-t^e)\prod_{a\in A}(1-t^a)}.$$

By adapting, without changes, the proof of \cite[Theorem 2.1]{CG} one gets $$\tilde{\chi}(T'_m)=\chi(Q,m)+\chi(\overline{D},m)$$ for all $m\in S$ and the new formula $$P=P_S=P_{\mathcal{G}_Q}+P_{\mathcal{G}_{\bar{D}}}.$$

Comparing with \textit{(2)} in Theorem \ref{theorem: teorema 2} with in \cite[Theorem 3.1]{CG}, one gets the expression of $P_{\mathcal{G}_{\bar{D}}}$ in terms of the $P_{\mathcal{G}_{D^J}}$

$$P_{\mathcal{G}_{\bar{D}}}=\sum_{\substack{J\subset E \\ \#J\geq 2}}(-1)^{\#J-1}P_{\mathcal{G}_{D^J}}.$$

See \cite{CG} for detailed homology computations when they are applied to homologies instead of Euler characteristics.

\section{Alexander duals}

Consider a nonempty finite set $E$ and a simplicial subcomplex $T$ of the simplex $\mathcal{P}(E)$ consisting of all subsets (parts) of $E$ with set of vertices $E$. The subset 
$$E'=supp(T)=\{e\in E\, | \, \{e\}\in T\}$$ is called the support of $T$, and the respective simplicial complexes 
\begin{align*}
T^\vee & =\{J\subset E'\, |\, E'\backslash J \notin T\}\\
\overline{T}^\vee & =\{L\subset E \, |\, E\backslash L \notin T\}
\end{align*}
are called the Alexander dual and the relative Alexander dual of $T$. The properties of the Alexander duality we will need are summarized in the following result (see also \cite{BM}).

\begin{proposition}
With assumptions and notations as above, on has:
\begin{enumerate}
\item $T$ and $T^\vee$ are subcomplexes of the simplex $\mathcal{P}(E')$, $\overline{T}^\vee$ is a subcomplex of $\mathcal{P}(E)$ and $T^\vee=\overline{T}^\vee\cap\mathcal{P}(E')$. The support of $T^\vee$ can be a proper subset of $E'$.
\item Reciprocity. $\overline{(\overline{T}^\vee)}^\vee=T$
\item $-\tilde{\chi}(T)=(-1)^{\# supp (T)}\tilde{\chi}(T^\vee)=(-1)^{\#E}\tilde{
\chi}(\overline{T}^\vee)$
\item If $T=\{\emptyset\}$, one assumes $\tilde{\chi}(T^\vee)=1$ by convention in order to satisfy \textit{(3)}.
\item $dim \, T \leq \# supp (T) -1\quad and \quad dim\,T^\vee \leq \# supp (T)-3$
\item Combinatorial duality. The simplicial homology $\tilde{H}_{\ell}(\overline{T}^\vee)$ and cohomology $\tilde{H}^{\#E-\ell-3}(T)$ $\mathbb{Z}$-modules are isomorphic for all $\ell$.
\end{enumerate}
\end{proposition}
\begin{proof}
Statement\textit{(1)} follows from definitions and \textit{(2)} from the equivalence $$L\in \overline{(\overline{T}^\vee)}^\vee\quad \iff \quad E\backslash L \in \overline{T}^\vee \quad \iff \quad E\backslash (E\backslash L)=L\in T.$$

To prove \textit{(3)} first notice that 
$$\overline{T}^\vee=\{K\cup (E\backslash E')\, |\, K\in T^\vee\} \cup \{C\cup F \, |\, C\subset E',\, F\subset E\backslash E', \ F\neq E\backslash E'\}.$$

Now, one has 
\begin{align*}
\tilde{\chi}(\overline{T}^\vee) & =  (-1)^{\# E\backslash E'}\sum_{K\in T^\vee}(-1)^{\# K-1}+(\sum_{C\subset E}(-1)^{\# C})(\sum_{\substack{F\subset E\backslash E'\\ F\neq E\backslash E'}}(-1)^{\#F-1}) \\
 & = (-1)^{\#E -\# E'}\tilde{\chi}(T^\vee)
\end{align*}

as $$\sum_{C\subset E'}(-1)^{\# C}=(1-1)^{\# E'}=0.$$

On the other hand
\begin{align*}
\tilde{\chi}(T^\vee) & =\sum_{\substack{J\subset E\\ J \notin T}}(-1)^{\#E'-\#J-1} = (-1)^{\#E'-1}\sum_{J\in T}(-1)^{\#J-1}\\
& = (-1)^{\#E'-1}\tilde{\chi}(T)
\end{align*}

as $$\sum_{J\in T}(-1)^{\# J}+\sum_{\substack{J\notin T\\ J\subset E'}}(-1)^{\# J}=(1-1)^{\#E'}=0.$$

Both equalities show the result. 

Statement \textit{(4)} considers the special case in which $T$ consists only of the empty face, in that case $T^\vee$ is not defined as $E'$ is empty too. However, $\overline{T}^\vee$ is defined and it is a ($\#E-2$)-topological sphere, so $\tilde{\chi}(\overline{T}^\vee)=(-1)^{\#E-2}=(-1)^{\#E}$. Then one needs to define $\tilde{\chi}(T^\vee)=1$ in order for \textit{(3)} to be true again.

To see \textit{(5)}, notice that $K\in T^\vee$ implies $\#(E'-K)\geq 2$ as $E'$ is the support of $T$. Thus, the dimension of the face $K$ needs to be $\#E'-3$ at most. 

Finally, statement \textit{(6)} is the central result of Alexander duality of simplicial complexes, and an elementary proof can be found, for instance, in \cite{BM}. The statement \textit{(6)} is true not only for the ring $\mathbb{Z}$, but in fact it is true for any commutative and with unit ring of values.
\end{proof}

Coming back to semigroups, the statements of the proposition apply to the complexes $T_m$ and their duals $T_m^\vee$ and relative duals $\overline{T}_m^\vee$, taking into account that $supp(T_m)=supp(m)$. From it, the main result of this section is directly deduced.

\begin{theorem}
Let $S$ be a finitely generated positive cancelative commutative semigroup and $E\subset S\backslash \{0\}$ a nonempty finite subset. Then one has the following two formulas for its Poincaré series
$P=P_S\in \hat{R}.$

\begin{enumerate}[itemsep=5mm]
\item $$P=\sum_{m \in S}(-1)^{\# supp(m)}\tilde{\chi}(T_m^\vee)t^m$$
\item $$P=\sum_{m\in S}(-1)^{\# E}\tilde{\chi}(\overline{T}_m^\vee)t^m$$
\end{enumerate}
\end{theorem}

\begin{proof}
It follows from \textit{(3)} in the previous proposition.
\end{proof}

Both formulas in above theorem for $P$ are in terms of combinatorial objects. In particular, they become practical when the dual complexes, in particular $T_m^\vee$ , can be easily visualized. For instance, notice that, for $\#E'\leq 6$, \textit{(4)} shows that $dim( T_m^\vee)\leq \#E'-3\leq 3$. Next, we will detail the topological possibilities for $T_m^\vee$ for $\#E'=0,\, 1,\, 2,\, 3,\, 4.$ 

\begin{figure}[h]
\includegraphics[width=70mm]{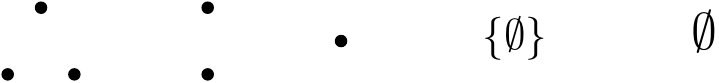}
\caption{}
\label{figure: duals}
\end{figure}

First, $\#E'=0$ only occur for $m\in Q$. The dual $T_m^\vee$ is not defined, but, as said before, the value $\tilde{\chi}(T_m^\vee)$ is defined and equal to 1. If $\#E'=1$, $E'=\{e\}$ then $T_m$ is one point and $T_m^\vee=\emptyset$. If $\#E'=2$, $T_m$ is either two isolated points or one segment, in the first case $T_m^\vee=\{\emptyset\}$, in the second $T_m^\vee=\emptyset$. If $\#E'=3$, the five types for $T_m$ are listed among those in figure \ref{figure: two} and have respective duals $T_m^\vee$ listed in figure \ref{figure: duals}.

Below, we list the 20 possible topological types for $T_m$ and $T_m^\vee$ when $\#E'=4$. We also list the values for $-\tilde{\chi}(T_m)=\tilde{\chi}(T_m^\vee)$. For all lists, one can visualize that $dim (T_m^\vee)\leq \#E'-3.$ 

\vspace{5mm}
\begin{figure}[h]
\includegraphics[width=\linewidth]{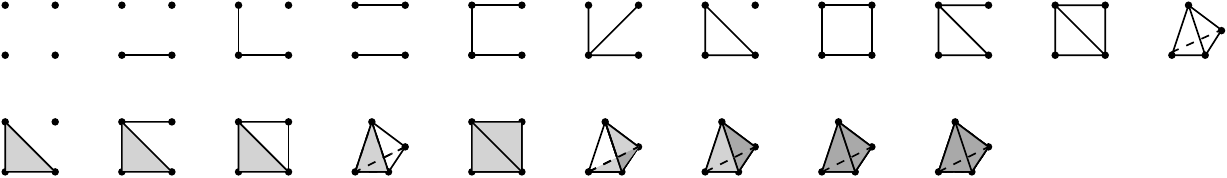}
\caption{}
\label{Figure TM}
\end{figure}

Figure \ref{Figure TM} is the list of the 20 possible topological types for $T_m$ where the last two are the border of the tetrahedron (2-sphere) and the full tetrahedron respectively.

\vspace{5mm}
\begin{figure}[h]
\includegraphics[width=\linewidth]{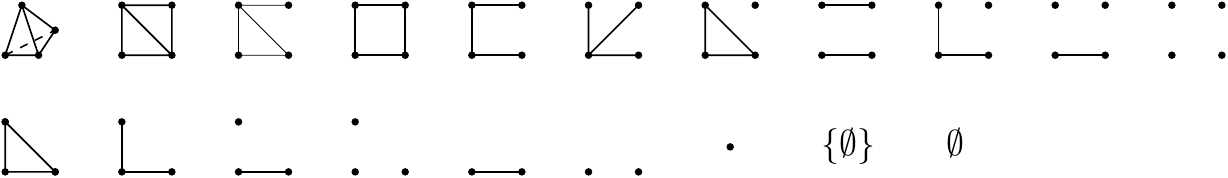}
\label{Figure TMV}
\caption{List of the corresponding $T_m^\vee$.}
\end{figure}

\begin{figure}[h]
\includegraphics[width=\linewidth]{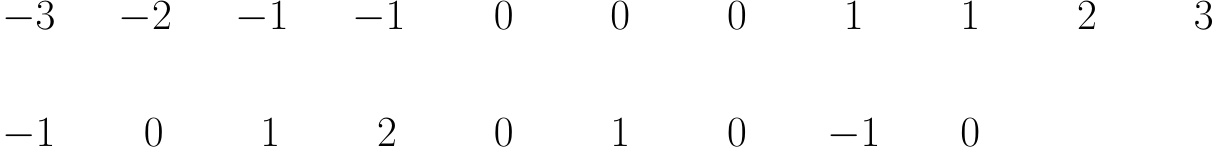}
\label{Figure NUM}
\caption{List of corresponding values of $-\tilde{\chi}(T_m)=\tilde{\chi}(T_m^\vee)$}
\end{figure}

\section{Syzygies}
For a simplicial subcomplex $T$ of the full simplex $\mathcal{P}(E)$ and any field $\mathbb{K}$, let us consider now not only its reduced Euler characteristic $\tilde{\chi}$ but also its reduced homology dimensions $\tilde{h}_j(T),\, -1\leq j\leq \#E-2$, $$\tilde{\chi}(T)=\sum_{j=1}^{\#E-2}(-1)^j \tilde{h}_j(T).$$ We exclude $j=\#E-1$, i.e., $T=\mathcal{P}(E)$ since its homology vanishes.

Although the integer $\tilde{\chi}(T)$ does not depend on the field $\mathbb{K}$, the homology dimensions $\tilde{h}_j(T)$ can depend on the characteristic of $\mathbb{K}$.

A typical classical example is that of the triangulation of the real projective plane in figure \ref{figure: triangulacion}, for which $h_2(T)=1$ for characteristic $\neq 2$ and $h_2(T)=0$ for characteristic 2. Notice that the real projective plane is a nonorientable surface.

\begin{figure}[h]
\includegraphics[width=30mm]{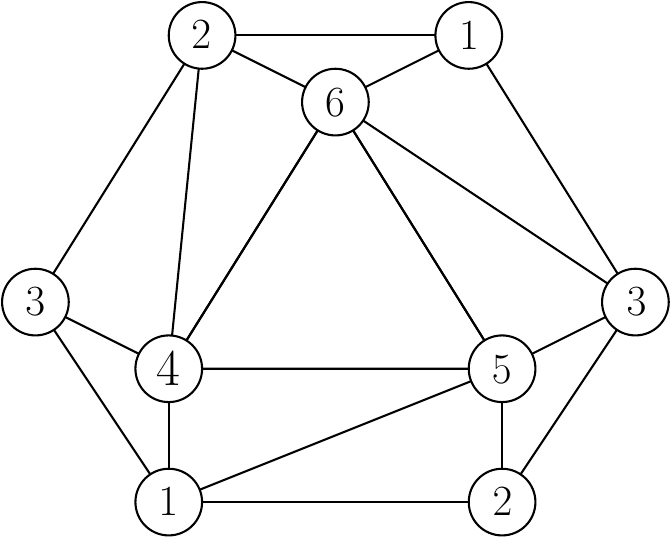}
\label{figure: triangulacion}
\caption{Triangulation of the real projective plane.}
\end{figure}

On the other hand, the homology type of every simplicial complex $T$ with $E' = E$ can be realized as that of one $T_m$ for some concrete semigroup $S$ and $m\in S$, as shown in \cite{BH}. To see it, take the subsemigroup $S$ of $\mathbb{Z}^{\#E+1}$ generated by the elements of type $(0,...,0,1,0,...,0)$ where the last coordinate is 0 and the integer 1 is the coordinate of the label of a vertex of $T$ for a chosen labeling, and by elements of type $(0,1,...,1,0,...,1)$ where 1 is the last coordinate and the other coordinates 1 are those of the vertices of a maximal face of $T$. Then, one can easily check that the element $m=(1,...,1)\in \mathbb{Z}^{\#E +1}$ belongs to $S$ and $T_m$ has the homotopy type of $T$. 

Once the field $\mathbb{K}$ is considered, other elements of $S$ related to the Poincaré series are the degrees of the syzygies of the minimal resolution of the graded semigroup $\mathbb{K}$-algebra $$\mathbb{K}[S]=R\otimes_\mathbb{Z}\mathbb{K}=\bigoplus _{m\in S}\mathbb{K}t^m$$ with grading on $S$. $\mathbb{K}[S]$ is seen as a module on the polynomial ring $$R_E=\mathbb{K}[X_e,\, e\in E]$$ where $X_e$ is a variable for each $e$ and the module structure comes from the $S$-algebra map sending $X_e$ to $t^e$.

The $S$-graded minimal resolution exists when the choice of $E$ generates the cone $C(S)$, as $\mathbb{K}[S]$ becomes a finite $R_E$-module minimally generated by $t^q$ for $q\in Q$. It is of type $$0\rightarrow R_E^{b_s}\rightarrow \cdots \rightarrow R_E^{b_0}\rightarrow R_E^{b_{-1}}\rightarrow \mathbb{K}[S]\rightarrow 0, \quad b_{-1}=\#Q$$ where $b_j,\, -1\leq j \leq s$ are the Betti numbers $b_j>0$. Its length $s$ is called the homological dimension and, according to the Auslander-Buchbaum Theorem, its complement $r=\#E-s$ is nothing but the depth of $\mathbb{K}[S]$. For simplicity, we will assume that $E$ generates all the semigroup $S$, i.e., $Q=\{0\}$, $b_{-1}=1$.

One always has $1\leq r \leq d$. When $r=d$, the algebra $\mathbb{K}[S]$ is said to be Cohen-Macaulay, and when $r=d$ and $b_s=1$ it is called Gorenstein. Moreover, the maps in the resolution are $S$-graded of degree 0, so one has well defined $S$-graded Betti numbers $b_{j,m}\in S$ for $0\leq j \leq s$ and $m\in S$, only finitely many of them being nonzero and one has $$b_j=\sum_{m\in S}b_{j,m}.$$

The graded Betti numbers are directly related to our combinatorial  tools by the equalities $$b_{i+1,m}=\tilde{h}_i(T_m)$$ for $-1\leq i \leq s-1$ and $m\in S$. All above, as well as the computation of the $S$-graded syzygies, is developed in detail in \cite{CG}, \cite{ CP}, \cite{BCPV}. In the Gorenstein case, one gets exactly one $g\in S$ for which $b_{s,g}\neq 0$ and, in fact, $b_{s,g}=1$. For such a $g$, two elements $m,m'$ with $m+m'=g$ are considered to be symmetric one to each other. 

Notice that all above numbers, except the value of the depth $r$ of $\mathbb{K}[S]$, depend on the choice of $E$. However $r$ depends only on $\mathbb{K}[S]$, although it depends on the characteristic of $\mathbb{K}$, as shown in  \cite{TH} by means of another historical example based on the same triangulation of the real projective plane. In particular, the following results are now easy to check.

\begin{proposition}
With assumptions and notations as above one has
\begin{enumerate}
\item $p=1-\sum_{j=0}^s(-1)^j \sum_{m\in S}b_{j,m}t^m$ is a polynomial
\item The following conditions are equivalent:
\begin{enumerate}
\item The algebra $\mathbb{K}[S]$ is Gorenstein
\item One has $b_{j,m}=b_{s-j,m'}$ for all pairs $m,m'$ symmetric with respect to some concrete $g\in S$. In particular $b_{-1}=b_{s,g}=1$
\item For some $g\in S$, $p$ satisfies the equation $p+(-1)^{\#E+d-1}t^g \overline{p}=0$ where $\overline{p}(t)=p(1/t)$.
\end{enumerate}
\end{enumerate}
\end{proposition}

\begin{remark}
For general choices of $E$ generating the cone $C(S)$, the independent term of the polinomial in \textit{(1)} is $\# Q$. We have assumed that $E$ generates $S$, thus $Q=\{0\}$, $N_Q=1$. Statements \textit{(1), (2b), (2c)} for a choice $E$ generating the cone $C(S)$ can be deduced from the corresponding to a larger choice $E\cup A$  which generates the semigroup $S$.

If an element $g$ in conditions \textit{(2b)} and \textit{(2c)} exists, then it is unique and the same for both. 

Condition \textit{(2c)} extends, for $d>1$, the usual condition for the symmetry of numerical semigroups. In fact, in that case, one can easily see that $f=g-e_E$ does not depend on $E$ and one has $$\prod
_{e\in E}(1-t^e)(P+t^f \overline{P})=0,$$
where $\overline{P}(t)=P(1/t)$, $c=f+1$ is the conductor of $S$.
\end{remark}

The semigroups of complete intersections are Gorenstein. For the choice of $E$ generating the semigroup $S$ one has $Q=\{0\}$, and the kernel of the first map of the resolution $R_E\rightarrow \mathbb{K}[S]$ is generated by $b_0\geq \#E-d$ $S$-homogeneous polynomials of degrees $m$ for which $b_{0,m} \neq 0$, properly counted. When $b_0=\#E-d$, $\mathbb{K}[S]$ is said to be a complete intersection. In that case, denote by $C$ the set of those degrees, each $m$ repeated $b_{0,m}$ times, and $g=\sum_{c\in C}c$. 

The minimal $S$-graded resolution can be realized as the Koszul complex of the regular sequence given by those $\#E-d$ generators. So, the degrees of the $(j+1)$-th syzygies are the possible sums of $j$ elements of $C$ properly counted, thus, \textit{(1)}, in the proposition shows  $$p=\prod_{c\in C}(1-t^c).$$

This $p$ satisfies the functional equation $$p+(-1)^{\#E+d-1}t^g \overline{p}=0$$ of the statement \textit{(2c)} of the proposition. The Poincaré series $P\in \hat{R}$ is given by a cyclotomic type expression $$P=\dfrac{\prod_{c\in C}(1-t^c)}{\prod_{e\in E}(1-t^e)}$$ and $f=\sum_{c\in C}c-\sum_{e\in E}e=g-e_E$. As far as we know, seeing whether the reciprocal is true is an open question.

Finally, we will show some relations of the depth $r$ and the Poincaré series. Assume again here that $E$ generates the cone $C(S)$. In \cite[Theorem 4.1]{CG}, it is proved that $r$ is the integer satisfying the following two conditions 
\begin{align*}
& \tilde{h}_{\#E-r}(T_m)=0\text{ for all } m\in S \\
& \tilde{h}_{\#E-r-1}(T_m)\neq 0\text{ for some } m\in S. 
\end{align*}
By the Alexander combinatorial duality, above conditions become equivalent to 
\begin{align*}
& \tilde{h}_{r-3}(\overline{T}_m^\vee)=0 \text{ for all }  m\in S\\
& \tilde{h}_{r-2}(\overline{T}_m^\vee) \neq 0\text{ for some }  m\in S
\end{align*}
and corresponding similar conditions for the $T_m^\vee$.

That is useful in practice. For instance, if $\#E=d$, the Cohen- Macaulay condition $r=d$ is that the $T_m$ are connected or $\{\emptyset\}$. For $r=1$, the condition says that some $T_m$ is a sphere. For $r=2$, it is equivalent to that no $T_m$ is a sphere, i.e. no $\overline{T}_m^\vee$ is $\{\emptyset\}$, and some $\overline{T}_m^\vee$ is not connected. For $r=3$ is equivalent to all $\overline{T}_m^\vee$ are connected or $\{\emptyset\}$ and $\tilde{h}_1(\overline{T}_m^\vee)\neq 0$ for some of them.

Above, by sphere, we mean a $(\#E-2)$-dimensional sphere. Also the dual complexes can be visualized in terms of the Apéry set $Q_E$  of the single choice $\{e_E\}$. In fact, one has $J\in \overline{T}_m^\vee$ iff $m+e_J\in Q_E$.

Above results are reformulations of known results in terms of the Alexander duals of the combinatorial tools. Bellow, in theorem \ref{theorem: 5.3} we will give new results in terms of the set theoretical tools of our semigroups.

Next, we will relate the depth $r$ with the key set theoretical tools sets of the semigroup. 

For it, let $r'$ be the integer defined by the two conditions:
\begin{align*}
&D^J=\emptyset\text{ for all }  J\subset E\quad \#J\geq \#E-r'+2\\
&D^J \neq \emptyset\text{ for some }  J\subset  E\quad \#J=\#E-r'+1.
\end{align*}

The integer $r'$ only depends on the sets $D^J$ and not on $\mathbb{K}$, and its relation to $r$ is given in the theorem below.  To stay it, we need a new definition. For a $k$-dimensional sphere of a simplicial subcomplex $T$ with support set $E$, $-1\leq k\leq \#E-2$ we mean a non face $J\notin T$ such that $J\backslash \{e\}\in T$ for all $e\in J$ and $\#J=k+2$.

\begin{theorem}
\label{theorem: 5.3}
With assumptions as above, one has
\begin{enumerate}
\item $r\geq r'$
\item If some $T_m$ has a $(\#E-r'-1)$-sphere, then $r'=r$. 
\item If $D$ is a finite set, then $r'=r$.
\end{enumerate}
\end{theorem}
\begin{proof}
For $m\in S$ consider the full simplex $\Sigma_m=\mathcal{P}(supp (T_m))$, and the $\mathbb{K}$-vector space isomorphisms $\tilde{H}_{\ell+1}(\Sigma_m,T_m)\cong \tilde{H}_{\ell}(T_m)$ given by the homology long exact sequence.

First, take $\ell=\#E-r'$. The $\ell+1=(\#E -r'+1)$-dimensional chain space for $(\Sigma_m,T_m)$ is generated by the subsets $J\subset E$ with $\#J=\#E-r'+2$ and $m\in D^J$. By definition of $r'$ one has $D^J=\emptyset$, so $$\tilde{h}_{\#E-r'+1}(\Sigma_m,T_m)=h_{\#E-r'}(T_m)=0$$ for all $m\in S$. Now, by the properties of $r$, one deduces $r\geq r'$. This shows \textit{(1)}.

Second, take $\ell=\#E-r'-1$ and a $(\#E-r'-1)$-sphere $J$ of some $T_m$. Then, $\#J=\#E-r'+1$. $J$ is a relative homology chain which, by construction, it is a cycle and not a border, so $\tilde{h}_{\#E-r'}(\Sigma_m, T_m)=\tilde{h}_{\#E-r'-1}(T_m)\neq 0$. This shows $r'=r$ as required in \textit{(2)}.

To prove \textit{(3)}, take $J\subset E$, $\#J=\#E-r'+1$, such that $D^J\neq \emptyset$, and an element $m\in D^J$. One has $J\subset supp(m)$ and $m-e_J\notin S$. If $J$ is a $(\#E-r'-1)$-sphere of $T_m$, the result follows from \textit{(2)}. Otherwise, there exists $e\in supp(m)$ such that $m-e_J+e=m+e-e_J\notin S$. Now continue with $m+e$ instead of $m$. If $J$ is a $(\#E-r'-1)$-sphere of $T_{m+e}$ the result follows from \textit{(2)}, if not repeat the argument. This procedure generates a increasing sequence of complexes $T_n$ with $J\notin T_n$ and $J\subset supp(n)$. Since $D$ is finite, $J$ needs to be a $(\#E-r'-1)$-sphere of $T_n$ for some $T_n$ in the sequence. This completes the proof of the theorem applying again \textit{(2)}.
\end{proof}

\begin{remark}
\text{(2)} stands also for cells. Requirements of \textit{(2)} and \textit{(3)} in the theorem for having $r=r'$ look rather strong. In fact $r=r'$ is impossible without requirements as $r$ can depend on the field characteristic, but $r'$ does not depend on it. However, the existence of a $(\#E-r'-1)$-sphere of \textit{(2)} could be verified in concrete cases and should be investigated in further others. 

For example if $r=2$, $\#E=4$, then $r'=1,2$ as $r'\leq r$. The conditions defining $r$ are $\tilde{h}_{-1}(\overline{T}_m^\vee)=0$ for every $m$ and $\tilde{h}_0(\overline{T}_m^\vee) \neq 0$ for some $m$, and the condition \textit{(2)} is the existence of a $1$-sphere if $r'=2$ and a $2$-sphere if $r'=1$. Our figures in section 4 show how all non connected $\overline{T}_m^\vee$ with less than 4 vertices have a dual $T_m$ containing the border of a non full triangle or a quadrangle, i.e. a $1$-sphere. So $r=r'$ in that case, i.e. $r'=2$ as we knew.

Other example is $\#E=d=r$, whose conditions are $\tilde{h}_0(T_m)=0$ for all $m$ and $\tilde{h}_{-1}(T_m)\neq 0$ for some $m$. The second one shows directly that a $(-1=\#E-r'-1)$-sphere of one $T_m$ exists, but this also follows from the fact that $Q\neq \emptyset$. This is the condition in \textit{(2)} for the case $r=r'$, in general $r'\leq r$. In this case one has $D^J=\emptyset$ for every $J\subset E$, $\#J\geq 2$, i.e. $D=\emptyset$. In particular, one has that $m-e\in S$, $m-e'\in S$, $e\neq e'$ implies $m-e-e'\in S$ for the semigroup. Notice that the first condition in the definition of $r'$ can be also stated as $D^J=\emptyset$ for all $J$ with $\#J=\#E-r'+2$, as this implies that $D^J=\emptyset$ also for larger cardinality of $J$. On the other hand, since sets $D^J$ are only defined for $\#J\geq 2$, the second condition in the definition of $r'$ in assumed by convention in the case $r'=\#E$, i.e. $D=\emptyset$.

To finish the paper, assume a double choice $E$ and $E\cup A$. the first one generating the cone $C(S)$ and the second one generating the semigroup $S$. The syzygies Poincaré series $P^h$ is defined as the element in $R[v]$ given by $$P^h=1-\sum_{j=0}^s \sum_{m\in S}b_{j,m}v^jt^m$$ where $b_{j,m}$ are the graded Betti numbers of the minimal resolution of $\mathbb{K}[S]$ as before. Again, we take the second choice for the sake of simplicity. One also has $$P^h=\sum_{i=-1}^{s-1}\sum_{m\in S}\tilde{h}_i(T_m)v^it^m.$$ It is clear that, evaluating $P^h$ for $v=-1$, one gets $P$. 
\end{remark}

The main results of \cite{CG} show that also $P^h$ is determined, not only by the combinatorial tools, but also by the set theoretical ones. Thus, we have the following result.

\begin{proposition}
\begin{enumerate}
\item $P^h$ can be computed, in linear algebra terms, from the complete homologies of the colored graph $\mathcal{G}_D$ and the homologies of the colored graphs $\mathcal{G}_Q$. 
\item The complete homologies of $\mathcal{G}_D$ can be computed, in linear algebra terms, from the homologies of the colored graphs $D^J$.
\end{enumerate}
\end{proposition}
\begin{proof}
It follows from Theorem 2.1, Popositions 3.1 and 3.2 in  \cite{CG}.
\end{proof}

\bibliographystyle{amsalpha}

\end{document}